\documentclass[reqno, 12pt]{amsart}

\usepackage{amsmath,amssymb,amsthm,amsfonts,mathrsfs,bbold}
\usepackage{fullpage}
\usepackage{xcolor}
\usepackage[colorlinks=true,
linkcolor=blue,
anchorcolor=blue,
citecolor=red
]{hyperref}
\allowdisplaybreaks 
\newtheorem{theorem}{Theorem}[section]
\newtheorem{lemma}[theorem]{Lemma}

\newtheorem*{conjecture*}{Conjecture}
\theoremstyle{definition}

\theoremstyle{remark}

\usepackage{colonequals}

\author{Runbo Li}
\address{The High School Affiliated to Renmin University of China, Beijing 100080, People's Republic of China}
\email{carey.lee.0433@gmail.com}

\makeatletter
\@namedef{subjclassname@2020}{\textup{}2020 Mathematics Subject Classification}
\makeatother

\title[]{On some problems of primes with the floor function}

\begin{document}

\begin{abstract}
Let $\left[x\right]$ be the largest integer not exceeding $x$. For $0<\theta \leq 1$, let $\pi_{\theta}(x)$ denote the number of integers $n$ with $1 \leq n \leq x^{\theta}$ such that $\left[\frac{x}{n}\right]$ is prime and $S_{\mathbb{P}}(x)$ denote the number of primes in the sequence $\left\{\left[\frac{x}{n}\right]\right\}_{n \geqslant 1}$. In this paper, we obtain the asymptotic formula
$$
\pi_{\theta}(x)=\frac{x^{\theta}}{(1-\theta) \log x}+O\left(x^{\theta}(\log x)^{-2}\right)
$$
provide that $\frac{435}{923}<\theta<1$, and prove that 
$$
S_{\mathbb{P}}(x)=x\sum_{p} \frac{1}{p(p+1)}+O_{\varepsilon}\left(x^{435/923+\varepsilon}\right)
$$
for $x \rightarrow \infty$. Thus improve the previous result due to Ma, Wu and the author.
\end{abstract}

\maketitle

\section{Introduction}

The investigations of the summations related to rounding up certain arithmetic functions are quite popular in recent years. It seems that this new term wave of enthusiasm starts from the paper of Bordellès, Dai, Heyman, Pan and Shparlinski \cite{BDHPS2019}, where the following asymptotic formula
\begin{equation}
\sum_{n \leqslant x} f\left(\left[\frac{x}{n} \right]\right)=x \sum_{n=1}^{\infty} \frac{f(n)}{n(n+1)}+O_{f}\left(x^{1 / 2+\varepsilon}\right)
\end{equation}
is given, provided that $f$ satisfies a broad condition involving the growth of the magnitude of it. Here, as usual $\varepsilon$ denotes an arbitrary small positive number. An example of particular interest is the one for $f=\Lambda$. In \cite{LWY2022}, Liu, Wu and Yang proved the following elaborate asymptotic formula
\begin{equation}
\sum_{n \leqslant x} \Lambda\left(\left[\frac{x}{n}\right]\right)=x \sum_{n=1}^{\infty} \frac{\Lambda(n)}{n(n+1)}+O\left(x^{9 / 19+\varepsilon}\right).
\end{equation}
In 2023, Li and Ma improved the exponent $\frac{9}{19}$ to $\frac{17}{36}$ in \cite{LI2023} and Zhang further improved the exponent to $\frac{435}{923}$ in \cite{ZhangWei}.\\

In an other direction, the work of Bordellès, Luca, Moree and Shparlinski \cite{BLMS2018} about the Bernoulli polynomials has led to the consideration of truncated distribution of the primes represented by $\left[x/n\right]$. Let $0<\theta \leq 1$ be a real number and $\pi_{\theta}(x)$ be the number of integers $n$ with $1 \leqslant n \leqslant x^{\theta}$ such that $\left\lfloor\frac{x}{n}\right\rfloor$ is prime. In \cite{MCW2019}, Ma, Chen and Wu proved
\begin{equation}
\pi_{\theta}(x)
=\begin{cases}
\sum_{k=1}^{L} \frac{(-1)^{k-1}(k-1) !}{(1-\theta)^{k}} \frac{x^{\theta}}{(\log x)^{k}}+O\left(\frac{x^{\theta}}{(\log x)^{L+1}}\right), & \frac{23}{47}<\theta<1, \\
x\sum_{p} \frac{1}{p(p+1)}+O\left(x^\frac{26}{53} (\log x)^\frac{119}{53}\right), & \theta=1,\end{cases}
\end{equation}
where $L \geqslant 1$ is any given integer. It is amazing that the leading term in the asymptotic formula of $\pi_{\theta}(x)$ is not continuous at the point $\theta=1$ when $x$ is a given large number. Ma, Chen and Wu further proposed the following conjecture:

\begin{conjecture*} 
For any $0<\theta<1$ and given integer $L \geqslant 1$, we have
\begin{equation}
\pi_{\theta}(x)=\sum_{k=1}^{L} \frac{(-1)^{k-1}(k-1) !}{(1-\theta)^{k}} \frac{x^{\theta}}{(\log x)^{k}}+O\left(\frac{x^{\theta}}{(\log x)^{L+1}}\right) .
\end{equation}
\end{conjecture*}
In 2022, Zhou and Feng \cite{ZF} proved that for any $\frac{9}{19}<\theta<1$ and $L \geqslant 1$, the conjecture is true. In 2023, Li \cite{LRB} showed that even $\frac{17}{36}<\theta<1$ is admissible.\\

In this paper, we shall give an improvement of the result obtained by Zhou, Feng and Li. We obtained this result by combine the methods in \cite{ZhangWei} and \cite{ZF}. Now, let's state our main result as the following theorem.

\begin{theorem}\label{t1} Let $\theta$ be a number with $\frac{435}{923}<\theta<1$ and $L \geqslant 1$ be a given integer. Then

\begin{equation}
\pi_{\theta}(x)=\sum_{k=1}^{L} \frac{(-1)^{k-1}(k-1) !}{(1-\theta)^{k}} \frac{x^{\theta}}{(\log x)^{k}}+O\left(\frac{x^{\theta}}{(\log x)^{L+1}}\right),
\end{equation}

\noindent where the implied constant depends on $\theta, L$ and the real number $\varepsilon>0$ which is contained in Lemma~\ref{sigma_delta}. 
\end{theorem}

By the same method in \cite{ZF}, we also proved this theorem.

\begin{theorem}\label{2} Let $\theta$ be a number with $\frac{435}{923}<\theta<1$. For any any integer $A \geqslant 1$, We have

\begin{equation}
\Lambda_{\theta}(x):=\sum_{n \leq x^{\theta}} \Lambda\left(\left[\frac{x}{n}\right]\right)=x^{\theta}+O\left(x^{\theta}(\log x)^{-A}\right),
\end{equation}

\noindent where the implied constant depends on $\theta, L$ and the real number $\varepsilon>0$ which is contained in Lemma~\ref{sigma_delta}.
\end{theorem}

Let $\mathbb{P}$ be the set of all primes and let $\mathbb{P}_{\text {w}}$ be the set of all prime powers. Denote by $\mathbb{1}_{\mathbb{P}}$ and $\mathbb{1}_{\mathbb{P}_{\text {w}}}$ their characteristic functions, respectively. In 2021, Heyman \cite{Heyman} proposed to study the number of primes or prime powers in the sequence $\left\{\left[\frac{x}{n}\right]\right\}_{n \geqslant 1}$ :

$$
S_{\mathbb{P}}(x):=\sum_{n \leqslant x} \mathbb{1}_{\mathbb{P}}\left(\left[\frac{x}{n}\right]\right), \quad S_{\mathbb{P}_{\text {w}}}(x):=\sum_{n \leqslant x} \mathbb{1}_{\mathbb{P}_{\text {w}}}\left(\left[\frac{x}{n}\right]\right) .
$$

Theorems 5 and 7 of \cite{Heyman} can be stated as follows:
\begin{align}
S_{\mathbb{P}}(x) & =C_{\mathbb{1}_{\mathbb{P}}} x+O\left(x^{1 / 2}\right), \\
S_{\mathbb{P}_{\text {w}}}(x) & =C_{\mathbb{1}_{\mathbb{P}_{\text {w}}}} x+O\left(x^{1 / 2}\right),
\end{align}

where $C_{\mathbb{1}_{\mathbb{P}}}:=\sum_{p} \frac{1}{p(p+1)}$ and $C_{\mathbb{1}_{\mathbb{P}_{\text {w}}}}:=\sum_{p, \nu \geqslant 1} \frac{1}{p^{\nu}\left(p^{\nu}+1\right)}$. \\

In 2021, Ma and Wu \cite{MAWU} proposed better results by breaking the $\frac{1}{2}$-barrier in the error term. They proved that for any $\varepsilon >0$,
\begin{align}
S_{\mathbb{P}}(x) & =C_{\mathbb{1}_{\mathbb{P}}} x+O_{\varepsilon}\left(x^{9 / 19+\varepsilon}\right), \\
S_{\mathbb{P}_{\text {w}}}(x) & =C_{\mathbb{1}_{\mathbb{P}_{\text {w}}}} x+O_{\varepsilon}\left(x^{9 / 19+\varepsilon}\right),
\end{align}
as $x \rightarrow \infty$, where the implied constants depend on $\varepsilon$.

In this paper, we shall give an improvement of the result obtained by Ma and Wu. Now let's state our result as the following theorem.

\begin{theorem}\label{t3}
For any $\varepsilon >0$,
\begin{align}
S_{\mathbb{P}}(x) & =C_{\mathbb{1}_{\mathbb{P}}} x+O_{\varepsilon}\left(x^{435 /923+\varepsilon}\right), \\
S_{\mathbb{P}_{\text{w}}}(x) & =C_{\mathbb{1}_{\mathbb{P}_{\text {w}}}} x+O_{\varepsilon}\left(x^{435 / 923+\varepsilon}\right),
\end{align}
as $x \rightarrow \infty$, where the implied constants depend on $\varepsilon$.
\end{theorem}

\section{Some Lemmas}

From now on, let $x$ be a large positive number. Let $\varepsilon>0$ be an arbitrary small positive number which may not be the same throughout our paper. Let $\mathbb{N}$ and $\mathbb{P}$ be the set of positive integers and prime numbers, respectively. The notation $p$ will always denote a prime number. Let $\pi(x)$ be the number of primes up to $x$ and

$$
\Lambda(n)=\left\{\begin{array}{ll}\log p & \text { if}\ n=p^{\alpha}, \\
0 & \text { otherwise}. \end{array}
\right.
$$

\noindent be the von Mangoldt function. For any real number $t$, let
$$
\rho(t)=t-\left[t\right]-1 / 2.
$$

\noindent For $0<D \leqslant x, D<t \leqslant 2 D$ and $\delta \notin-\mathbb{N}$, let
$$
\Sigma_{\delta}(x, D, t)=\sum_{D<d \leqslant t} \Lambda(d) \rho\left(\frac{x}{d+\delta}\right).
$$

We need some auxiliary results before the proof of Theorems~\ref{t1} and ~\ref{t3}.\\

\begin{lemma}\label{sigma_delta}
Let $\delta \notin-\mathbb{N}$ be a fixed constant. For $D<t \leqslant 2 D$ and $D<x^{2 / 3}$, we have

$$
\Sigma_{\delta}(x, D, t) \ll_{\varepsilon}\left\{\begin{array}{ll}
x^{1 / 2+\varepsilon} D^{-1 / 6} & \text { if }\ D<x^{3 / 7}, \\
x^{1 / 3+\varepsilon} D^{2 / 9} & \text { if }\ x^{3 / 7} \leqslant D<x^{6 / 13}, \\
x^{1 / 6+\varepsilon} D^{7 / 12} & \text { if }\ x^{6 / 13} \leqslant D<x^{482/923}, \\
x^{435/923+\varepsilon} & \text { if }\ x^{482/923} \leqslant D<x^{488/923}.
\end{array} \right.
$$
\end{lemma}

\begin{proof} From {\cite[Proposition 4.1]{LWY2022}} with $(\kappa, \lambda)=\left(\kappa^{\prime}, \lambda^{\prime}\right)=(1 / 2,1 / 2)$ and {\cite[(3.6)]{ZhangWei}} with $(\kappa, \lambda)=\left(\kappa^{\prime}, \lambda^{\prime}\right)=(1 / 2,1 / 2)$, $\varrho=6 / 923$ and $\varpi=20.5 / 923$, we have
$$
\Sigma_{\delta}(x, D, 2 D) \ll_{\varepsilon} x^{\varepsilon}\left(x^{1 / 6} D^{7 / 12}+D^{5 / 6}+x^{1 / 3} D^{2 / 9}+x^{1 / 2} D^{-1 / 6}\right)
$$
for $D<x^{3 / 4}$ and
$$
\Sigma_{\delta}(x, D, 2 D) \ll_{\varepsilon} x^{435/923+\varepsilon}
$$
for $x^{482/923} \leqslant D<x^{488/923}$. In fact, by carefully checking the proof of {\cite[Proposition 4.1]{LWY2022}} and {\cite[(3.6)]{ZhangWei}}, we still have
\begin{equation}
\Sigma_{\delta}(x, D, t) \ll_{\varepsilon} x^{\varepsilon}\left(x^{1 / 6} D^{7 / 12}+D^{5 / 6}+x^{1 / 3} D^{2 / 9}+x^{1 / 2} D^{-1 / 6}\right)
\end{equation}
for $D<x^{3 / 4}$ and
\begin{equation}
\Sigma_{\delta}(x, D, t) \ll_{\varepsilon} x^{435/923+\varepsilon}
\end{equation}
for $x^{482/923} \leqslant D<x^{488/923}$. The lemma then follows from direct discussions.
\end{proof}

For $D \leq x$ and $\delta \notin-\mathbb{N}$, let
$$
\mathscr{S}_{\delta}(x, D)=\sum_{D<p \leqslant 2 D} \rho\left(\frac{x}{p+\delta}\right).
$$

\begin{lemma}\label{sigma2} Let $\delta \notin-\mathbb{N}$ be a fixed constant. For $D<x^{2 / 3}$, we have

$$
\mathscr{S}_{\delta}(x, D) \ll_{\varepsilon} \left\{\begin{array}{ll}x^{1 / 2+\varepsilon} D^{-1 / 6}+D^{1 / 2} & \text { if }\ D<x^{3 / 7}, \\ x^{1 / 3+\varepsilon} D^{2 / 9}+D^{1 / 2} & \text { if }\ x^{3 / 7} \leqslant D<x^{6 / 13}, \\ x^{1 / 6+\varepsilon} D^{7 / 12}+D^{1 / 2} & \text { if }\ x^{6 / 13} \leqslant D<x^{482/923}, \\ x^{435/923+\varepsilon}+D^{1/2} & \text { if }\ x^{482/923} \leqslant D<x^{488/923}.\end{array} \right.
$$
\end{lemma}

\begin{proof}
For $0<D \leq x, D<t \leq 2 D$ and $\delta \notin-\mathbb{N}$, let
$$
\mathscr{G}_{\delta}(x, D, t)=\sum_{D<p \leqslant t} \vartheta(d) \rho\left(\frac{x}{p+\delta}\right)
$$
where
$$
\vartheta(n)= \left\{\begin{array}{ll}\log p\ & \text{if}\ n=p \text{ is\ a\ prime},  \\ 
0\ & \text { otherwise}.\end{array}
\right.
$$
Note that
\begin{equation}
\sum_{D<d \leqslant t} \Lambda(d) \rho\left(\frac{x}{d+\delta}\right)=\sum_{D<d \leqslant t} \vartheta(d) \rho\left(\frac{x}{d+\delta}\right)+O\left(t^{1 / 2}\right).
\end{equation}
so we have
\begin{equation}
\mathscr{G}_{\delta}(x, D, t)=\Sigma_{\delta}(x, D, t)+O\left(D^{1 / 2}\right)
\end{equation}
for any $D<t \leq 2 D$. Integrating by parts, we have
\begin{equation}
\mathscr{S}_{\delta}(x, D)=\frac{\mathscr{G}_{\delta}(x, D, 2 D)}{\log 2 D}+\int_{D}^{2 D} \frac{\mathscr{G}_{\delta}(x, D, t)}{t(\log t)^{2}} d t .
\end{equation}
Then our lemma follows from Lemma~\ref{sigma_delta} by routine computations.
\end{proof}

\begin{lemma}\label{3}{\cite[Lemma 4]{ZF}}. Let $\theta$ be a positive number with $0<\theta<1$ and $L \geqslant 1$ be a given integer. Then

$$
x \sum_{p \geqslant x^{1-\theta}} \frac{1}{p(p+1)}=\sum_{k=1}^{L} \frac{(-1)^{k-1}(k-1) !}{(1-\theta)^{k}} \frac{x^{\theta}}{(\log x)^{k}}+O\left(\frac{x^{\theta}}{(\log x)^{L+1}}\right),
$$

where the implied constant depends only on $L$ and $\theta$.
\end{lemma}

\begin{lemma}\label{4}{\cite[Proposition 3.1]{C2019}}. Let $f$ be a positive-valued function on $\mathbb{N}$ and $D$ a parameter with $D \leq x$. Then,

$$
\sum_{D<n \leqslant x} f\left(\left[\frac{x}{n}\right]\right)=\sum_{d \leqslant x / D} f(d) \sum_{x /(d+1)<n \leqslant x / d} 1+O\left(f\left(\frac{x}{D}\right)\left(1+\frac{D^{2}}{x}\right)\right).
$$
\end{lemma}

\begin{lemma}\label{5}{\cite[Lemma 6]{ZF}}. Let $\theta$ be a positive number with $0<\theta<1$ and $A$ is any given positive number, we have

$$
x \sum_{d \geqslant x^{1-\theta}} \frac{\Lambda(d)}{d(d+1)}=x^{\theta}+O\left(x^{\theta}(\log x)^{-A}\right),
$$

where the implied constant depends only on $A$ and $\theta$.
\end{lemma}

\section{Proof of Theorem 1.1}

\noindent \textbf{Proof of Theorem 1.1.} For $\theta>\frac{435}{923}$, we spilt the sum $\pi_{\theta}(x)$ into the following two shorter sums
\begin{equation}
\pi_{\theta}(x)=S_{1}+S_{2}
\end{equation}
where
$$
S_{1}=\sum_{\substack{n \leqslant x^{435/923}, \left[ x / n\right] \in \mathbb{P}}} 1 \quad and \quad S_{2}=\sum_{\substack{x^{435/923}<n \leqslant x^{\theta},  \left[ x / n\right] \in \mathbb{P}}} 1.
$$
Trivial estimate leads to the bound
\begin{equation}
S_{1} \leqslant \sum_{n \leqslant x^{435/923}} 1 \leqslant x^{435/923}.
\end{equation}
By Lemma~\ref{4}, we can rewrite $S_{2}$ as
\begin{align}
\nonumber S_{2} &=\sum_{x^{1-\theta} \leqslant p \leqslant x^{488/923}} \sum_{x /(p+1)<n \leqslant x / p} 1+O\left(x^{2 \theta-1}\right)\\
\nonumber &=\sum_{x^{1-\theta} \leqslant p \leqslant x^{488/923}}\left(\frac{x}{p}-\rho\left(\frac{x}{p}\right)-\frac{x}{p+1}+\rho\left(\frac{x}{p+1}\right)\right)+O\left(x^{2 \theta-1}\right) \\
&=x \sum_{p \geqslant x^{1-\theta}} \frac{1}{p(p+1)}-x \sum_{p>x^{488/923}} \frac{1}{p(p+1)}+R_{1}(x)-R_{0}(x)+O\left(x^{2 \theta-1}\right),
\end{align}
where
$$
R_{\delta}(x)=\sum_{x^{1-\theta} \leqslant p \leqslant x^{488/923}} \rho\left(\frac{x}{p+\delta}\right) \quad(\delta=0\text { or } 1) .
$$
It is easy to see that $x^{2 \theta-1} \ll x^{\theta-\varepsilon}$ for any $\theta<1$ and it is clear that
\begin{equation}
x \sum_{p>x^{488/923}} \frac{1}{p(p+1)} \leqslant x \sum_{n \geqslant x^{488/923}} \frac{1}{n(n+1)} \ll x^{435/923}.
\end{equation}
From Lemma~\ref{3}, we have
\begin{equation}
x \sum_{p \geqslant x^{1-\theta}} \frac{1}{p(p+1)}=\sum_{k=1}^{L} \frac{(-1)^{k-1}(k-1) !}{(1-\theta)^{k}} \frac{x^{\theta}}{(\log x)^{k}}+O\left(\frac{x^{\theta}}{(\log x)^{L+1}}\right) .
\end{equation}
To complete the proof of our theorem, it remains to show that
\begin{equation}
R_{\delta}(x) \ll x^{\theta}(\log x)^{-(L+1)}
\end{equation}
for $\delta=0$ and 1. For any positive integer $i$, let $D_{i}=x^{488/923} 2^{-i}$. Since $\theta>435/923$, then $D_{i} \leqslant x^{488/923}<x^{2 / 3}$ for all $1 \leqslant i \leqslant\left[\frac{\theta-435/923}{\log 2} \log x\right]+1$. By Lemma~\ref{sigma2},
$$
\left|R_{\delta}(x)\right| \leqslant \sum_{1 \leqslant i \leqslant\left[\frac{\theta-435/923}{\log 2} \log x\right]+1} \mathscr{S}_{\delta}\left(x, D_{i}\right)
$$
$$
\ll_{\varepsilon} \sum_{1 \leqslant i \leqslant\left[\frac{\theta-435/923}{\log 2} \log x\right]+1}\left(x^{1 / 2+\varepsilon} D_{i}^{-1 / 6}+x^{1 / 3+\varepsilon} D_{i}^{2 / 9}+x^{1 / 6+\varepsilon} D_{i}^{7 / 12}+x^{435/923+\varepsilon}+D_{i}^{1/2}\right)
$$
\begin{equation}
\ll_{\varepsilon} x^{(\theta+2) / 6}+x^{435/923+\varepsilon} \ll_{\varepsilon, \theta} x^{435/923+\varepsilon} \ll_{\varepsilon, \theta} x^{\theta}(\log x)^{-(L+1)},
\end{equation}
valid for $\frac{435}{923}<\theta<\frac{764}{923}$. Combined with the range $\frac{17}{36}<\theta<1$ of Li \cite{LRB}, we get the theorem.\\

The proof of Theorem~\ref{2} is similar to the proof of Theorem 1 by replacing Lemma~\ref{3} with Lemma~\ref{5}.

\section{Proof of Theorem 1.3}

\noindent \textbf{Proof of Theorem 1.3.} Let $f=\mathbb{1}_{\mathbb{P}}$ or $\mathbb{1}_{\mathbb{P}_{\text {w}}}$. First we write

\begin{equation}
S_{f}(x)=\sum_{n \leqslant x} f\left(\left[\frac{x}{n}\right]\right)=S_{f1}(x)+S_{f2}(x)
\end{equation}

with

$$
S_{f1}(x):=\sum_{n \leqslant x^{435/923}} f\left(\left[\frac{x}{n}\right]\right), \quad S_{f2}(x):=\sum_{x^{435/923}<n \leqslant x} f\left(\left[\frac{x}{n}\right]\right) .
$$
Trivial estimate leads to the bound
\begin{equation}
S_{f1}(x) \ll x^{\frac{435}{923}}.
\end{equation}
By Lemma~\ref{4}, we can rewrite $S_{f2}$ as
\begin{align}
\nonumber S_{f2}(x) & =\sum_{d \leqslant x^{488/923}} f(d) \sum_{x /(d+1)<n \leqslant x / d} 1 \\
\nonumber & =\sum_{d \leqslant x^{488/923}} f(d)\left(\frac{x}{d}-\rho\left(\frac{x}{d}\right)-\frac{x}{d+1}+\rho\left(\frac{x}{d+1}\right)\right) \\
& =x \sum_{d \geqslant 1} \frac{f(d)}{d(d+1)}+R_{1}^{f}(x)-R_{0}^{f}(x)+O(x^{435/923}),
\end{align}
where we have used the following bounds
$$
x \sum_{d>x^{488/923}} \frac{f(d)}{d(d+1)} \ll x^{435/923}, \quad \sum_{d \leqslant x^{435/923}} f(d)\left(\rho\left(\frac{x}{d+1}\right)-\rho\left(\frac{x}{d}\right)\right) \ll x^{435/923}
$$
and
$$
R_{\delta}^{f}(x)=\sum_{x^{435/923}<d \leqslant x^{488/923}} f(d) \rho\left(\frac{x}{d+\delta}\right).
$$
Combining (25), (26) and (27), it follows that
$$
S_{f}(x)=x \sum_{d \geqslant 1} \frac{f(d)}{d(d+1)}+O_{\varepsilon}\left(\left|R_{1}^{f}(x) \right|+\left|R_{0}^{f}(x)\right|+x^{435/923}\right) .
$$
On the other hand, we have
$$
R_{\delta}^{\mathbb{1}_{\mathbb{P}_{\text {w}}}}(x)=\sum_{x^{435/923}<p^{\nu} \leqslant x^{488/923}} \rho\left(\frac{x}{p^{\nu}+\delta}\right)=R_{\delta}^{\mathbb{1}_{\mathbb{P}}}(x)+O\left(x^{244/ 923}\right) .
$$
Thus in order to prove Theorem~\ref{t3} , it remains to show that
\begin{equation}
R_{\delta}^{\mathbb{1}_{\mathbb{P}}}(x) \ll_{\varepsilon} x^{435/923+\varepsilon} \quad(x \geqslant 1)
\end{equation}
for $\delta=0$ and 1. By using Lemma~\ref{sigma2},
\begin{align}
\nonumber R_{\delta}^{\mathbb{1}_{\mathbb{P}}}(x)& \ll_{\varepsilon} x^{\varepsilon} \max _{x^{435 / 923}<D \leqslant x^{488 / 923}} \mathscr{S}_{\delta}(x, D)\\
\nonumber & \ll_{\varepsilon}x^{\varepsilon}\max _{x^{435 / 923}<D \leqslant x^{482 / 923}} \mathscr{S}_{\delta}(x, D)+x^{\varepsilon}\max _{x^{482 / 923}<D \leqslant x^{488 / 923}} \mathscr{S}_{\delta}(x, D)\\
\nonumber & \ll_{\varepsilon} \max _{x^{435 / 923}<D \leqslant x^{482 / 923}} \left(x^{1 / 6+\varepsilon} D^{7 / 12}+D^{1/2}\right)+\max _{x^{482 / 923}<D \leqslant x^{488 / 923}} \left(x^{435/923+\varepsilon}+D^{1/2}\right)\\
\nonumber & \ll_{\varepsilon} x^{435/923+\varepsilon}.
\end{align}
This completes the proof of Theorem~\ref{t3}.

\bibliographystyle{plain}
\bibliography{bib}
\end{document}